\documentclass{amsart}
\usepackage{graphicx}
\usepackage{amssymb}
\usepackage{mathrsfs}
\usepackage[all]{xy}
\DeclareMathOperator{\ad}{ad}
\DeclareMathOperator{\ev}{ev}

\DeclareMathOperator{\id}{id}
\DeclareMathOperator{\Aut}{Aut}
\DeclareMathOperator{\eq}{eq}
\DeclareMathOperator{\Bun}{Bun}
\DeclareMathOperator{\cBun}{cBun}

\theoremstyle{plain}
\newtheorem{theorem}{Theorem}[section]
\newtheorem*{theorem*}{Theorem}

\newtheorem{proposition}[theorem]{Proposition}

\theoremstyle{definition}
\newtheorem{definition}[theorem]{Definition}

\theoremstyle{remark}
\newtheorem{example}{Example}[section]

\newtheorem{remark}{Remark}[section]

\numberwithin{figure}{section}

\newcommand{\cF}{{\mathcal F}}
\newcommand{\cC}{{\mathcal C}}
\newcommand{\cK}{{\mathcal K}}
\newcommand{\cH}{{\mathcal H}}
\newcommand{\cR}{{\mathcal R}}
\newcommand{\cS}{{\mathcal S}}

\newcommand{\fg}{{\mathfrak g}}
\newcommand{\fk}{{\mathfrak k}}

\newcommand{\RR}{{\mathbb R}}
\newcommand{\ZZ}{{\mathbb Z}}

\renewcommand{\a}{\alpha}
\renewcommand{\b}{\beta}

\renewcommand{\d}{\delta}

\newcommand{\LGS}{LG\rtimes_\rho S^1}

\newcommand{\<}{\langle}
\renewcommand{\>}{\rangle}
\newcommand{\ds}{\displaystyle}

\newcommand{\vp}{\vphantom{{-1}}}
\newcommand{\wt}{\widetilde}
\newcommand{\ol}{\overline}

\begin{document}

\title[Circle actions]{Circle actions, central extensions \\ and string structures}
  \author[M. K. Murray]{Michael K. Murray}
  \address[M. K. Murray]
  {School of Mathematical Sciences\\
  University of Adelaide\\
  Adelaide, SA 5005 \\
  Australia}
  \email{michael.murray@adelaide.edu.au}

 \author[R. F. Vozzo]{Raymond F. Vozzo}
  \address[R. F. Vozzo]
  {School of Mathematical Sciences\\
  University of Adelaide\\
  Adelaide, SA 5005 \\
  Australia}
  \email{raymond.vozzo@adelaide.edu.au}
  \date{\today}

\subjclass[2010]{22E67, 53C08, 81T30}

\thanks{
The authors acknowledge the support of the Australian Research Council and useful discussions with Alan Carey, Danny Stevenson and Mathai Varghese. The second author acknowledges the support of an Australian Postgraduate Research Award.}

\begin{abstract}
The  caloron correspondence can be understood as an equivalence of categories between $G$-bundles over circle bundles and $LG \rtimes_\rho S^1$-bundles where $LG$ is the group of smooth loops in $G$.  We use it, and lifting  bundle gerbes, to derive an explicit differential form  based formula for the (real) string class of an 
$LG \rtimes_\rho S^1$-bundle. 
\end{abstract}

\maketitle

\tableofcontents

\section{Introduction}

The caloron correspondence was first introduced in \cite{Garland:1988} as a bijection between
isomorphism classes of $G$-instantons on $\RR^3 \times S^1$ (calorons) and $\Omega G$-monopoles on $\RR^3$, where
$\Omega G$ is the group of based loops in $G$.  The motivation in that case  was the study of monopoles for 
loop groups, in particular, their twistor theory. It was subsequently \cite{Garland:1989} applied to the  case of instantons 
on the four-sphere and the four-sphere minus a two-sphere and loop group monopoles on hyperbolic three-space.
The motivation for the present work however was   \cite{Murray:2003}, which used the caloron 
correspondence to relate string structures on loop group bundles and the Pontrjagyn class of $G$-bundles.
In particular it calculated an explicit de Rham representative for Killingback's string class \cite{Killingback:1987} using
bundle gerbes.  In \cite{MurVoz} we followed a similar approach   to define higher classes
of $\Omega G$-bundles which we called string classes and discussed their properties.   The present work generalises 
\cite{Murray:2003} in a different direction replacing the space $M \times S^1$ by a principal $S^1$-bundle 
$Y \to M$ and deriving a formula for the image  in de Rham cohomology of the string class of an $LG \rtimes_\rho S^1$-bundle $P \to M$ where
$LG$ is the space of free loops in $G$ and $LG \rtimes_{\rho} S^1$ the semi-direct product where the circle
acts by rotating the loop.

We begin by discussing the caloron correspondence in this broader context. 
This was first done in  \cite{Bergman:2005} and  later in \cite{Bouwknegt:2009} in a string theory
context.  It is a correspondence between 
$G$-bundles $\widetilde P \to Y$ over a circle bundle $Y \to M$ and $LG \rtimes_\rho S^1$-bundles $P \to M$.  It proves useful to introduce an intermediate step which are  $\rho$-equivariant $LG$-bundles $\overline P \to Y$.   We also consider the action of the caloron correspondence on connections which necessitates the introduction of Higgs fields.  

With the caloron correspondence complete we review  basic material on lifting problems for principal bundles and then 
introduce the central extension of $LG \rtimes_\rho S^1$ proving that it must have the form $\widehat{LG} \rtimes_{\hat\rho} S^1$
where $\widehat{LG} \to LG$ is the Kac-Moody central extension of $LG$. 

The final sections review the notion of a lifting bundle gerbe and extend the results of  \cite{Murray:2003}
to the case of the string class of an  $LG \rtimes_\rho S^1$-bundle.   Our central results are Theorem 
\ref{T:LGxS^1string} and \ref{T:LGxS^1Pont}:

\begin{theorem*}
Let $P \to M$ be a principal $LG\rtimes_\rho S^1$-bundle and let $\Phi$ be a Higgs field for $P$ and $(A,a)$ be a connection for $P$ with curvature $(F,f).$ Then the real string class of $P$, is represented in de Rham cohomology by
$$
-\frac{1}{4\pi^2}\int_{S^1} \langle F+ f\Phi, \nabla\Phi \rangle \,d\theta,
$$
where
$$
\nabla\Phi = d\Phi + [A,\Phi] - \partial A - a\partial \Phi.
$$
\end{theorem*}

\begin{theorem*}
Let $P \to M$ be a principal $LG\rtimes_\rho S^1$-bundle and $\wt P \to Y \to M$ be the corresponding $G$-bundle over an $S^1$-bundle. Then the real string class of $P$ is given by the integration over the fibre of the first Pontrjagyn class of $\wt P.$ That is,
$$
s(P) = \int_{S^1} p_1 (\wt P).
$$
\end{theorem*}

Throughout this paper, $G$ will be a compact,  connected Lie group and all  cohomology groups will use real coefficients.

\section{The caloron correspondence for $LG \rtimes_\rho S^1$-bundles.}\label{S:caloron}

Recall that in \cite{MurVoz} we explained the caloron correspondence between $\Omega G$-bundles on a manifold $M$ and certain framed $G$-bundles on $M \times S^1$. We want to extend this to the case of $LG$-bundles and also replace $M \times S^1$ by a circle bundle $Y  \to M$.  To do this it is useful to first introduce a correspondence between semi-direct product bundles.

\subsection{Semi-direct correspondence}
 We need some results on equivariant bundles but first we need to make 
conventions for left and right circle actions.  One circle action we will be interested in will be that  on a right principal $S^1$-bundle.  However we will also want to 
consider actions of the circle on Lie groups $K$ arising from automorphisms $\rho \colon S^1 \to \Aut(K)$. 
In this case the natural thing to do is to make the action on the left.  So our convention will be that
groups are acted on on the left via homomorphisms to the automorphism group and spaces are acted on on the 
right. We then have

\begin{definition}
Let $\overline P \to M$ be a $K$-bundle and   $\rho \colon S^1 \to \Aut(K)$  a homomorphism. 
We say that $\overline P \to M$ is a \emph{$\rho$-equivariant bundle}
if $S^1$ acts on $\overline P$ covering an action on $M$ such that for all $p \in \overline P$,  $k \in K$ and $\theta \in S^1$ we have 
$R_\theta(pk) = R_\theta(p)\rho_{-\theta}(k)$.  Here $p \mapsto R_\theta(p)$ denotes the action 
of $\theta$ on $p$ and $k \mapsto \rho_\theta(k)$ the action of $\theta$ on $k$.
\end{definition}

\begin{remark}
As we will see in the next example the sign change here is a result of mixing left and right actions.
\end{remark} 

\begin{example}
\label{ex:loop_bundles}
If $X$ is a space let $LX $ be the space of all smooth maps  $\gamma \colon S^1 \to X$.  We make $S^1$ act on $LX$ by defining 
$R_{\theta}(\gamma)(\phi) = \gamma(\phi + \theta)$.  If $P \to M$ is a $G$-bundle we can form the $LG$-bundle $LP \to LM$ and 
note that this is a $\rho$-equivariant $LG$-bundle where $\rho \colon S^1 \to \Aut(LG)$ is the action $\rho_\theta(g)(\phi) = g(\phi - \theta)$. We call this a {\em loop bundle}.
\end{example}

\begin{example} 
The discussion above defines an action $\rho \colon S^1 \to LG$ by $\rho_\theta(g)(\phi) = g(\phi + \theta)$. 
There is a natural notion of a homomorphism of groups with actions of the circle on them. In particular if $\rho \colon S^1 \to G$ is an action consider the map $\hat\rho \colon G \to LG$ defined by $\hat\rho(g)(\phi) = \rho_{\phi}(g)$.
Then this is a homomorphism of groups and in fact a homomorphism of groups with circle actions as
$$
\hat\rho( \rho_\theta(g))(\phi) = \rho_\phi( \rho_\theta (g)) = \rho_{\theta+\phi}(g) = 
\hat\rho(g)(\theta + \phi) = \rho_\theta( \hat\rho(g))(\phi).
$$
\end{example}

Recall that if $S^1$ acts on a group $K$ we form the semi-direct product $K \rtimes_\rho S^1 = K \times S^1$ with the product
$$
(k, \theta)(h , \phi) = (k\rho_{\theta}(h), \theta + \phi).
$$
Note that there is a short exact sequence
$$
1 \to K \to K \rtimes_\rho S^1 \to S^1 \to 1
$$
so that if we have a principal $K \rtimes_\rho S^1$-bundle $P \to M$ it induces an $S^1$-bundle $P(S^1) \to M$.

\begin{proposition}[Semi-direct correspondence]
\label{prop:equiv-semi}
Let $Y \to M$ be an $S^1$-bundle and $\rho \colon S^1 \to \Aut(K)$ a homomorphism. Then there is a bijective correspondence between
\begin{enumerate}
\item Principal $K$-bundles $\overline P \to Y$ which are $\rho$-equivariant for the circle 
 action on $Y$, and
\item Principal $K \rtimes_\rho S^1$-bundles $P \to M$ with a circle bundle isomorphism from $P(S^1)$  to $Y$.
\end{enumerate}
\end{proposition}
\begin{proof}
Let $\overline P \to Y$ be $\rho$-equivariant so that we have a circle action on $\overline P$ covering the principal bundle
circle action on $Y$ such that $R_\theta(pk) = R_\theta(p)\rho_{-\theta}(k)$. Define an action 
of $K \rtimes_\rho S^1$ on the right of $\overline P$ by $p(k, \theta) = R_\theta(pk)$. Then we have
\begin{align*}
(p(k, \theta))(h, \phi) &= (R_\theta(pk))(h, \phi) \\
                         &=R_\phi(R_\theta(pk)h) \\
                         &= R_{\theta + \phi}(pk)\rho_{-\phi}(h)\\
                         &= R_{\theta+\phi}(p)\rho_{-\theta-\phi}(k) \rho_{-\phi}(h)\\
                         \end{align*}
                         and thus
\begin{align*} p( (k, \theta)(h, \phi) ) &= p (k\rho_\theta(h), \theta + \phi))\\
                                       &= R_{\theta+\phi}(p k \rho_\theta(h)) \\
                                       &= R_{\theta + \phi}(p) \rho_{-\theta - \phi}(k) \rho_{-\theta - \phi}(\rho_{\theta}(h))\\
                                       &= R_{\theta + \phi}(p) \rho_{-\theta - \phi}(k) \rho_{-\phi}(h)\\&= (p(k, \theta))(h, \phi) 
                                       \end{align*}
  so this is a right action.  Local triviality and freeness of the action are straightforward to check as is the fact that 
  the orbits are the fibres of the composed map $\overline P \to M$ making $\overline P$ a $K \rtimes_\rho S^1$-bundle over $M$.     
  The induced $S^1$-bundle is the bundle $\overline P \times_{K \rtimes S^1} S^1$ where equivalence 
  classes satisfy  $[p, \phi] = [p (k, \theta), - \phi + \theta]$.    Let $\pi \colon \overline P \to Y$ be the 
  projection and define $\chi \colon   \overline P \times_{G \rtimes S^1} S^1 \to Y$ by $\chi([p, \theta]) = \pi(\rho_\theta(p)) = \rho_\theta(\pi(p))$. 
  Then $\chi(\rho_\phi([p, \theta]) ) = \chi([p , \theta + \phi]) = \rho_{\phi+\theta}(\pi(p)) = \rho_\phi(\rho_\theta(\pi(p))) = 
  \rho_\phi(\chi([p, \theta]))$. So $\chi $ is an $S^1$-bundle isomorphism. 
                      
   Consider the converse. Define an action of $K$ on $P$ by composing with the projection $K \to K \rtimes S^1$,
   that is $p k = p(k, 0)$. It is straightforward to check that these are the fibres of the projection 
   $P \to P \times_{K \rtimes S^1} S^1$. Define a circle action on $P$ by $R_\theta(p) = p(1, \theta)$.
   We have 
 \begin{align*}
         R_\theta(p k) & = p(k, 0)( 1, \theta) \\
         &= p(1, \theta)(1, -\theta) (k, 0)(1, \theta) \\
         & = R_\theta(p) (\rho_{-\theta}(k), 0 ) \\
         & = R_\theta(p)\rho_{-\theta}(k)
         \end{align*}                  
    so that $P $ is a $\rho$-equivariant $K$-bundle $P \to P \times_{K \rtimes S^1} S^1$. Now just pull $P$ back by the inverse of the 
    $S^1$-equivariant isomorphism $\chi \colon  P \times_{K \rtimes S^1} S^1 \to Y$.                          
\end{proof}

\begin{remark} 
\label{rem:reduced} 
Notice that if $P \to M$ is a $K \rtimes_\rho S^1$-bundle over $M$ and 
$P(S^1)$ is trivial we can consider the subset of $P$ of all elements mapping to the image of the 
trivialisation. It is easy to see that this is a reduction of $P$ to $K$. Conversely if
$\overline P \to M$ is a $K$-bundle and we induce a $K \rtimes S^1$-bundle $P = \overline P \times_K (K \rtimes S^1)$ then 
it has $P(S^1)$ trivial and its  reduction to $K$ is naturally isomorphic to $\overline P$ via the map 
$p \mapsto [p, (1, 0)]$. 
\end{remark}

Denote by $\Bun_K^\rho/\Bun_{S^1}$ the category whose objects are triples $(\overline P, Y, M)$ 
where $M$ is a manifold, $Y \to M$  an $S^1$-bundle and  $\overline P \to Y$ a $\rho$-equivariant bundle and whose morphisms from $(\overline P, Y, M)$ to $(\overline P', Y', M')$ are triples $(f, g, h) $ where $g \colon Y \to Y'$ is an $S^1$-bundle map covering $h \colon M \to M'$ and 
$f \colon \overline P \to \overline P'$ is a $K$-bundle map covering $g $ and commuting with the circle action. 
Also denote by $\Bun_{K \rtimes_\rho S^1}$ the category of $K \rtimes_\rho S^1$-bundles
$P \to M$ with bundle maps as morphisms.  In both cases we have functors 
$$
\Pi_{S^1} \colon \Bun_K^\rho/\Bun_{S^1}  \to \Bun_{S^1} 
$$
and 
$$
\Pi_{S^1} \colon \Bun_{K \rtimes_\rho S^1}  \to \Bun_{S^1} 
$$
to the category of $S^1$-bundles. The first is the obvious projection and the 
second is the functor $\Pi_{S^1} (P) = P (S^1)$.  What we have described in the 
proposition above are two functors
$$
\cF \colon \Bun_K^\rho/\Bun_{S^1} \to \Bun_{K \rtimes_\rho S^1} 
$$
and 
$$
\cK \colon \Bun_{K \rtimes_\rho S^1}  \to  \Bun_K^\rho/\Bun_{S^1} 
$$
which in fact are pseudo-inverses. This means that there are natural isomorphisms
$$
\cK \circ \cF  \simeq \id_{\Bun_K^\rho/\Bun_{S^1}} \quad\text{and}\quad \cF \circ \cK \simeq \id_{\Bun_{K \rtimes_\rho S^1}} .
$$
so we have

\begin{proposition}
The semi-direct product correspondence defines an equivalence of categories
$$
\Bun_K^\rho/\Bun_{S^1} \to \Bun_{K \rtimes_\rho S^1}
$$
\end{proposition}

\subsection{Caloron correspondence}
If $Y \to M$ is an $S^1$-bundle denote by $L_{\eq}Y \subset LY$ those loops $f \colon S^1 \to Y$ which are $S^1$-equivariant in the sense that $ f(\theta + \phi) = R_\phi(f(\theta))$.
Define $\eta \colon  Y \to LY$ by $\eta(y)(\phi) = R_\phi(y)$ and note that 
$\eta(y)(\theta + \phi) = 
R_{\theta+\phi}(y) = R_\phi( R_\theta(y)) = R_\phi(\eta(y)(\theta))$.
Notice also that $\eta$ is equivariant in the sense that $\eta(R_\theta(y)) = R_\theta(\eta(y))$.
In fact we have $\eta(R_\theta(y))(\phi) = R_\phi(R_\theta(y)) = R_{\theta+\phi}(y) = \eta(y)(\theta+ \phi)
= R_\theta(\eta(y))(\phi)$.

Consider now a $G$-bundle $\widetilde P \to Y$ where $Y \to M$ is an $S^1$-bundle. We can then form the $\rho$-equivariant $LG$-bundle $L\widetilde P \to LY$.
If we pull this back with $\eta$ to $\eta^*(L\widetilde P) \to Y$ we obtain a $\rho$-equivariant $LG$-bundle over $Y$.  It is easy to check that this defines a functor
$$
\cR \colon \Bun_G/\Bun_{S^1} \to \Bun_{LG}^\rho/\Bun_{S^1}.
$$

This functor is also an equivalence which we can see by constructing a pseudo-inverse
$$
\cS \colon  \Bun_{LG}^\rho/\Bun_{S^1}  \to \Bun_G/\Bun_{S^1}
$$
If $\overline P \to Y$ is a $\rho$-equivariant $LG$-bundle we can use the evaluation 
homomorphism $\ev_0 \colon LG \to G$, whose kernel is the based loop group $\Omega G$, to form the associated $G$-bundle  $\overline P \times_{LG} G = \overline P  / \Omega G \to Y$ which is $\cS(\overline P)$. 

\begin{proposition}
\label{prop:pseudo-inverses}
The functors $\cR$ and $\cS$ are pseudo-inverses of each other.
\end{proposition} 
\begin{proof} 
Start with a $\rho$-equivariant $LG$-bundle $\overline P \to Y$ over a circle bundle $Y \to M$. Applying $\cS$
yields $\overline  P / \Omega G \to Y$. Applying $\cR$ we first construct $L(\overline  P / \Omega G ) \to LY$
and then pull back with $\eta \colon Y \to LY$. We want to construct an isomorphism 
$\overline  P \to \eta^*(L(\overline  P / \Omega G ))$ of $\rho$-equivariant $LG$-bundles over $Y$. 
We do this by defining  a $\rho$-equivariant map $\widehat \eta \colon \overline  P \to L(\overline  P / \Omega G ) $
covering $\eta$  by $\widehat \eta ( p)(\theta) =  R_\theta(p)\Omega G$.  We leave the reader to check that this defines an isomorphism of bundles which 
is moreover a natural transformation from $\cR \circ \cS$ to the identity functor on $\Bun_{LG}^\rho/\Bun_{S^1} $.

In the other direction we start with a $G$-bundle $\wt P \to Y$ over a circle bundle $Y$. First we form the 
$LG$-bundle $L\wt P \to LY$, pull-back with $\eta \colon Y \to LY$ and then  $(\cS \circ \cR)(P) = \eta^*(L\wt P) / \Omega G$. The evaluation map  $\ev_0 \colon \eta^*(L\wt P) \to \wt P$  is an isomorphism of $LG / \Omega G = G$-bundles giving the required result. Again 
we leave the requirements of a natural transformation for the reader to check.  
\end{proof}

We define the {\em caloron transform} to be the composition
of functors
$$
\cC  = \cS \circ \cK \colon  \Bun_{LG \rtimes_\rho S^1} \to \Bun_G/\Bun_{S^1} 
$$
and deduce that 

\begin{proposition}[Caloron correspondence]
The caloron transform is an equivalence of categories.
\end{proposition}

\begin{remark}
The caloron correspondence in this form first appeared in \cite{Bergman:2005} and later in \cite{Bouwknegt:2009}. We have introduced the equivariant bundles so we can proceed in a similar manner to   \cite{MurVoz} and  essentially reduce the inverse caloron transform to a pull-back which makes it simpler 
to define on connections.  
\end{remark}

It will be useful to write down a formula for the caloron transform of an $\LGS$-bundle $P \to Y$ directly. Note that this is given by the constructions outlined above as $\cC(P) = P \times_{LG} S^1$ where the $LG$ action on $P$ is factored through the map $LG \to \LGS$ (Proposition \ref{prop:equiv-semi} and the discussion preceding Proposition \ref{prop:pseudo-inverses}). The following Proposition gives an alternative description:

\begin{proposition}[\cite{Bouwknegt:2009, Vozzo:PhD}]\label{P:caloron direct}
The caloron transform of the $\LGS$-bundle $P \to Y$ is isomorphic to the $G$-bundle 
$$
\widetilde{P} = (P \times G \times S^1) / (\LGS),
$$
where the action of $\LGS$ on $P \times G \times S^1$ is given by 
$$
(p, g, \theta)(\gamma, \phi) = (p(\gamma, \theta), \gamma(\theta)^{-1}g, \theta - \phi)
$$
and the action of $h \in G$  on $[p,g,\theta] \in \widetilde P$ is given by  $[p,g,\theta]h = [p,gh,\theta]$.
\end{proposition}

\begin{proof}
In \cite{Bouwknegt:2009} and \cite{Vozzo:PhD} the caloron correspondence is developed using the bundle $\widetilde{P}$ above. Here we show that the map $\widehat \psi \colon P\times G \times S^1 \to \ol P $ given by $(p, g ,\theta) \mapsto R_\theta (p g) = p(g, \theta)$ descends to an isomorphism of $G$-bundles $\psi \colon \wt P \xrightarrow{\sim} \ol P / \Omega G.$ To see this, note that if we chose a different representative of the equivalence class of $(p,g,\theta)$, say $(p(\gamma, \phi), \gamma(\theta)^{-1}g, \theta -\phi)$, then 
\begin{align*}
\widehat \psi (p(\gamma, \phi), \gamma(\theta)^{-1}g, \theta -\phi) 
	&= p (\gamma, \phi) (\gamma(\theta)^{-1}g, \theta -\phi)\\
	&= p (\gamma \gamma(\theta)^{-1} g, \theta)\\
	&= p (g, \theta)(g^{-1} ,-\theta)(\gamma \gamma(\theta)^{-1} g, \theta)\\
	&= p (g, \theta) (g^{-1} \rho_{-\theta}(\gamma)\gamma(\theta)^{-1} g, 0)
\end{align*}
and $g^{-1} \rho_{-\theta}(\gamma)\gamma(\theta)^{-1} g$ is a based loop in $G$. Clearly $\psi$  commutes with the $G$ action and so is an isomorphism of $G$-bundles.

\end{proof}

We can summarise the results of this section with the following diagram:

$$
\xymatrix@C=10ex@R=3ex{\wt P \ar[dd]^G & && \ol P\ar[dd]^{LG} &&  \\
		&\ar[r]^{\cR} &&& \ar[r]^{\cF} && P\ar[dd]^{\LGS}\\
		Y\ar[dd]^{S^1} &&& Y\ar[dd]^{S^1} && \\
		&& \ar[l]_{\cS} &&& \ar[l]_{\cK}&  M\\
		M &&& M & }
$$

\subsection{Higgs fields and $LG \rtimes_\rho S^1$-bundles}

We want to first discuss how the semi-direct correspondence in Proposition \ref{prop:equiv-semi} can be extended to bundles with connections.  We find along the way that we need to introduce Higgs fields. 

\begin{definition}
Let $\rho \colon S^1 \to \Aut(K)$ be a homomorphism and $\overline P \to Y$ a $\rho$-equivariant bundle.  A connection 
$\bar A \in \Omega^1(\overline P, \fk)$ is called \emph{$\rho$-equivariant} if $R_{\theta}^*(\bar A) = \rho_{-\theta}(\bar A)$ where
we abuse notation and denote by $\rho_\theta \in \Aut(\fk)$ the automorphism of Lie algebras induced by $\rho_{\theta} \in  \Aut(K)$. 
\end{definition}

\begin{example}
\label{ex:loop_connection}
We have seen in Example \ref{ex:loop_bundles} that if $\widetilde P \to Y$ is a $G$-bundle then $L\widetilde P \to LY$ is a 
$\rho$-equivariant $LG$-bundle. If $\tilde A$ is a connection on $\wt P$ we can define a connection $L\tilde A$ on $L\wt P \to LY$ as follows. Let  $\gamma \in L\wt P$ and note that  a tangent vector to $L\wt P$ at $\gamma$ is a section 
of the tangent bundle to $\wt P$ pulled-back by $\gamma$. Let $\xi \in T_\gamma L\wt P  = \Gamma([0, 1], \gamma^*(T\wt P))$
we define $L\tilde{A}_\gamma(\xi) \in L\fg$ by $L\tilde{A}_\gamma(\xi) (\theta) = \tilde{A}_{\gamma(\theta)}(\xi(\theta))$. It is straightforward
to check that this is a $\rho$-equivariant connection. 
\end{example}

It is not obvious that $\rho$-equivariant connections exist but we will prove this below in the case of interest. 

We need some additional notation. Let $\delta(p) \in T_pP$ be the tangent vector to the circle action and let $\iota_p \colon \fk \to T_pP$ be the derivative of the map $k \mapsto R_k(p)$.   If $\gamma \colon S^1 \to K$   denote  by $\partial \gamma(\theta) \in T_{\gamma(\theta)} K$ the 
image of $\partial/\partial\theta$ under the tangent to $\gamma$.  Denote by $\gamma^{-1} \partial\gamma(\theta)$ the 
left-translate of this to $\fk$ so that $\gamma^{-1} \partial \gamma$ is a loop in $\fk$.  In particular $\theta \mapsto \rho_\theta(k)$ is defines a map $\rho_k \colon S^1 \to K$ and thus
$\rho_k^{-1} \partial \rho^{\vp}_k(0) \in \fk$.

\begin{example}
If $\gamma \in K = LG$ then  $ \rho_\gamma^{-1} \partial \rho^{\vp}_\gamma(0) = \gamma^{-1} \partial \gamma$.  
\end{example}

We also note here the data which corresponds to a connection on a $K \rtimes_\rho S^1$-bundle $P.$   It is straightforward using the definition of connection and the semi-direct product action to show that

\begin{proposition}\label{P:(A,a)}
A connection on a $K \rtimes_\rho S^1$-bundle $P \to M$ is equivalent to a pair $(A, a),$ where $A$ is $\fk$ valued and $a$ is $\RR$ valued, satisfying
\begin{align*}
a(\delta) &= 1 \\
A_p(\iota_p (\xi )) &= \xi \\
\end{align*}
and
\begin{align*}
R_k^*(A) &= \ad(k^{-1})A - a \rho_k^{-1} \partial \rho^{\vp}_k(0)\\
R_k^*(a) &=  a \\
R_{\theta}^*(A) &= \rho_{-\theta}(A) \\
R_{\theta}^*(a) &= a \\
\end{align*}

\end{proposition}

We also need the notion of a Higgs field.

\begin{definition}
\label{def:Higgsfield}
Let $P \to M$ be a $K \rtimes_\rho S^1$-bundle.  A {\em Higgs field} for $P$ is a function $\Phi \colon P \to \fk$
satisfying:
$$
\Phi(p(k, \theta)) = \rho_{-\theta}(\ad(k^{-1})\Phi(p) + \rho_k^{-1} \partial \rho^{\vp}_k(0) )
$$
for all $k \in K, \theta \in S^1$ and $p \in P$.
\end{definition}

\begin{proposition}
Higgs fields exist.
\end{proposition}
\begin{proof}
First notice that a convex combination of Higgs fields is a Higgs field. We can then combine locally 
defined Higgs fields with a partition of unity. So assume now that $P = M \times K \rtimes_\rho S^1$ 
and then define
$$
\Phi((k, \theta)) = \rho_k^{-1} \partial \rho^{\vp}_k(0)
$$
We leave it as an exercise to show that this is a Higgs field.
\end{proof}

\begin{proposition} 
\label{prop:equiv-semi-conn}
Let $Y \to M$ be an $S^1$-bundle with connection $\alpha$ and $\rho \colon S^1 \to \Aut(K)$ a homomorphism. Then there is a bijective correspondence between
\begin{enumerate}
\item Principal $K$-bundles $\ol P \to Y$ which are $\rho$-equivariant for the circle 
 action on $Y$ equipped with a $\rho$-equivariant connection $\bar A$ on $\ol P \to Y$ and
\item Principal $K \rtimes_\rho S^1$-bundles $P \to M$ with a circle bundle isomorphism from $P(S^1)$  to $Y$
and a pair $(( A, a), \Phi)$ consisting of a connection $(A, a) $ which projects to 
$\alpha$ on $Y$ under the isomorphism and a  Higgs field $\Phi$  on $P \to M$.
\end{enumerate}
\end{proposition}

\begin{proof}

Note that as a manifold $\ol P = P$. It is easy to show that  
$$
\bar A =  A + a \Phi
$$
defines a $\rho$-equivariant connection on $\ol P \to Y$. 

Conversely given the $\rho$-equivariant connection on $\ol P$ we can define a Higgs field $\Phi(p) = A_p(\delta_p)$
for $p \in P = \ol P$. Then the connection $(A, a)$ is defined by
$$
(A, a) = (\bar A + \pi^*(\alpha) \Phi, \pi^*(\alpha))
$$
where $\pi \colon \ol P \to Y$. 

\end{proof}

These constructions allow us to extend the functors defined in the previous section to act on bundles with  connections and Higgs fields.  We define a new category $\cBun_K^\rho/\cBun_{S^1} $ to consist
of objects $(\ol P, Y, M)$  in   $\Bun_K^\rho/\Bun_{S^1} $ with the addition of a connection $\a$ on $Y \to M$ and 
a $\rho$-equivariant connection $\bar A$ on  $\ol P \to Y$. Morphisms are just the morphisms in $\Bun_K^\rho/\Bun_{S^1} $
which preserve the connections. Similarly an object in $\cBun_{K \rtimes_\rho S^1}$ is an object $P \to M$ in $\Bun_{K \rtimes_\rho S^1}$ with a connection and Higgs field $((A,a), \Phi)$ again morpisms preserve connections and Higgs fields.  Then we have functors 
$$
\cF \colon \cBun_K^\rho/\cBun_{S^1} \to \cBun_{K \rtimes_\rho S^1} 
$$
and 
$$
\cK \colon \cBun_{K \rtimes_\rho S^1}  \to  \cBun_K^\rho/\cBun_{S^1} 
$$
which  are pseudo-inverses. 

To define the caloron transform for bundles with connections (and Higgs fields) we have to extend the second part of the correspondence
$$
\cR \colon \Bun_G/\Bun_{S^1} \to \Bun_{LG}^\rho/\Bun_{S^1}
$$
and
$$
\cS \colon  \Bun_{LG}^\rho/\Bun_{S^1}  \to \Bun_G/\Bun_{S^1}
$$
to act appropriately. First we need to define the corresponding categories. We define $\cBun_G/\cBun_{S^1} $ 
to consist of $G$-bundles  $\wt P \to Y$ with connection $\tilde A$ and a connection $\a$ 
on $Y \to M$.  Morphisms are those morphisms preserving the connections.  The definition of $\cBun_{LG}^\rho/\cBun_{S^1}$ follows from the definitions already given. 

To extend the correspondence note first that if we start with a $\rho$-equivariant  $LG$-bundle $\ol P  \to Y$ with a $\rho$-equivariant connection $ \bar A$ then there is an induced connection on the $G$-bundle $\ol P \times_{LG} G = \ol P / \Omega G$. 
In the other direction if $\wt P \to Y$ is a $G$-bundle and  $\tilde A$ is a connection on $\wt P \to Y$ and  $\a$ a conection on $Y \to M$ then there is a loop connection $L\tilde A$ on $L\wt P \to LY$ which is an $LG$ connection.
This pulls back to give a connection $\eta^*(L\tilde A)$. It remains to check that this is $\rho$-equivariant but as $\eta$
is $\rho$-equivariant it suffices to check this for $L\tilde A$ which we noted in Example \ref{ex:loop_connection} is straightforward.

It is now easy to define the inverse caloron correspondence with connection as it is the result of applying
Proposition \ref{prop:equiv-semi-conn} and noting that equivariant connections and Higgs fields pull-back.  
We have

\begin{proposition}\label{P:LGxS^1 connection correspondences}
The caloron correspondence extends to an equivalence of categories  between $G$-bundles with connection over $S^1$-bundles with connection and $\LGS$-bundles with connection and Higgs field.
\end{proposition}
\begin{proof}
We need to show the constructions in Proposition \ref{prop:pseudo-inverses} preserve the connections. Recall 
that if $Q \to X$ is a $LG$-bundle  then $Q/\Omega G = Q \times_{LG} G$ is a $G$-bundle. 
Moreover if $\ev_0 \colon LG \to G$ is the evaluation map whose kernel is $\Omega G$ and $p \colon Q \to Q/\Omega G$
is the projection then  given a connection one-form $A$ on $Q$ there is a unique connection one-form $B$ on $Q/\Omega G$ 
satisfying $p^*B = \ev_0(A)$ where $\ev_0 \colon L\fg \to \fg$ is the induced evaluation map on Lie algebras 
(see \cite{Kobayashi:1963} for a proof in a more general setting).

Recall from Proposition \ref{prop:pseudo-inverses} the constructions involved.
We start with a $\rho$-equivariant $LG$-bundle $\ol P \to Y$ over a circle bundle $Y \to M$. Then we  construct an isomorphism 
$\ol  P \to \eta^*(L(\ol  P / \Omega G ))$ of $\rho$-equivariant $LG$-bundles over $Y$
by defining  a $\rho$-equivariant map $\widehat \eta \colon \ol  P \to L(\ol  P / \Omega G ) $
covering $\eta$  by $\widehat \eta ( p)(\theta) =  R_\theta(p)\Omega G$.   Notice that we can lift 
$\widehat \eta$ to $\tilde \eta \colon \ol P \to L \ol P$ by letting $\tilde\eta ( p)(\theta) =  R_\theta(p)$.
Consider a $\rho$-equivariant connection $\bar A$ on $\ol P$.  Let $B$ be the induced connection on $\ol P / \Omega G$. Then if $p \colon \ol P \to \ol P / \Omega G$ we have $p^*B = \ev_0(\bar A)$  from the discussion above.
The connection on $L(\ol  P / \Omega G ) $ is $LB$ and we need to show that $\widehat\eta^*LB = \bar A$. 
We have 
$$
\widehat \eta^*LB = \tilde \eta^* p^* LB = \tilde\eta^*L ( \ev_0(\bar A)).
$$
Let $\xi $ be a tangent vector at $q \in \ol P$ and note that the $\rho$-equivariance of $\bar A$ implies that $R_\theta^* \bar A = \rho_{-\theta}(\bar A)$ 
so that 
$$
\bar A(R_\theta(p))(R_\theta(\xi))(\phi) = \bar A(p)(\xi)(\phi + \theta)
$$
so that 
$$ 
\bar A(R_\theta(p))(R_\theta(\xi))(0) = \bar A(p)(\xi)(\theta)
$$ 
and thus
\begin{align*}
\tilde\eta^*(L ( \ev_0(\bar A)))(p)(\xi)(\theta)  &= \bar A(R_\theta(p))(R_\theta(\xi))(0)\\
 &= \bar A(p)(\xi)(\theta)
\end{align*}
so that
$$
\widehat \eta^*LB = \tilde\eta^*L ( \ev_0(\bar A))  = \bar A.
$$
as required.

In the other direction we start with a $G$-bundle $\wt P \to Y$ over a circle bundle $Y$ and use the evaluation map  $\ev_0 \colon \eta^*(L\wt P)/\Omega G \to \wt P$  to define an isomorphism of $LG / \Omega G = G$-bundles.
If $\tilde A$ is a connection on $\wt P$ then $\eta^*L\tilde A$ is a connection 
on $\eta^*(L\wt P)$ and the connection $B$ on $\eta^*(L\wt P) / \Omega G$ satisfies $p^*B = \ev_0(\eta^*L\tilde A)$ where $p \colon \eta^*(L\wt P) \to \eta^*(L\wt P)/\Omega G$. Denoting $\widetilde \ev_0 \colon \eta^*(L\wt P) \to \wt P$  where $\widetilde\ev_0 = \ev_0 \circ p$ it suffices to show that 
$\widetilde\ev_0^*\tilde A = \ev_0(\eta^*L\tilde A)$.  If $\xi$ is a tangent vector to $\gamma \in L\wt P$ then $(\widetilde\ev_0^*\tilde A)(\gamma)(\xi) = 
\tilde A(\gamma(0))(\xi(0)) = \ev_0( \theta \mapsto \tilde A(\gamma(\theta))(\xi(\theta))) =  \ev_0(\eta^*L\tilde A)(\gamma)(\xi)$ as required.

\end{proof}

Recall that in Proposition \ref{P:caloron direct} we gave a formula for the caloron transform $\widetilde{P}$ of $P$. Namely, $\wt P = (P \times G \times S^1)/ (\LGS)$. We can also give a formula for the connection on this bundle induced by the functor $\cC$

\begin{proposition}\label{P:caloron connection}
Let $P$ be an $\LGS$-bundle with connection $(A,a)$ and Higgs field $\Phi$. Then the connection form on the bundle $\wt P$ (considered as a form on $P \times G \times S^1$ which descends to the quotient) is given by
$$
\tilde{A} = \ad(g^{-1}) A(\theta) + \Theta + \ad(g^{-1})\Phi(a + d\theta),
$$
where $\Theta$ is the Maurer-Cartan form on $G$.

\end{proposition}

\begin{proof}
Recall from Proposition \ref{P:caloron direct} that we have a bundle map
$$
\xymatrix@C=10ex{P \times G \times S^1 \ar[r]^{\widehat \psi} \ar[d]_{\tilde \pi} &\ol P \ar[d]^{\bar \pi}\\
		\dfrac{P \times G \times S^1}{\LGS}  \ar[r]^{\psi} & \ol  P / \Omega G}
$$
given by $\widehat \psi (p,g,\theta) = p(g, \theta).$ We know that the connection on the equivariant $LG$-bundle $\ol P$ is given in terms of the pair $((A,a), \Phi)$ on $P$ by $\bar A = A + a\Phi$ (Proposition \ref{prop:equiv-semi-conn}) and that the connection $B$ on $\ol P /\Omega G$ satisfies $\bar{\pi}^* B = \ev_0(\bar A) = \ev_0(A + a\Phi)$ (from the proof of Proposition \ref{P:LGxS^1 connection correspondences}). Therefore, since the diagram above commutes, to give the connection on $\wt P$ as a form on $P \times G \times S^1$ it is enough for us to calculate $\widehat{\psi}^* \ev_0(A + a\Phi)$. Note that if $Q$ is a principal $G$-bundle and $R \colon Q \times G \to Q$ is the right action of $G$ on $Q$ then for tangent vectors $X \in T_qQ$ and $g\xi \in T_gG$ (where $\xi \in \fg$ and here we are identifying $T_gG$ with $\fg$ via left multiplication on $G$) the push-forward $R_* \colon T_qQ \times T_gG \to T_{qg}Q$ is given (using the Liebnitz rule) by $(X, g\xi) \mapsto (R_g)_* X + \iota_{qg} \xi$. In this case, if $(X, g\xi , x_\theta)\in T_{(p,g,\theta)}(P\times G \times S^1)$ is a tangent vector then the derivative of $\hat \psi$ is given by
$$
\widehat{\psi}_* (X, g\xi , x_\theta) = (R_{(g,\theta)})_*X + \iota_{p(g,\theta)}(\xi, x),
$$
where $x\in \RR$ is the Lie algebra element corresponding to the tangent vector $x_\theta \in T_\theta S^1$. (Note that here we are using the fact that $\ol P$ is really the same as the bundle $P$.) We have
\begin{align*}
(\widehat{\psi}^* \ev_0(A+a\Phi)& )_{(p,g,\theta)}(X, g\xi, x_\theta)\\
	&= \ev_0(A+a\Phi)_{p(g,\theta)}((R_{(g,\theta)})_*X + \iota_{p(g,\theta)}(\xi, x))\\
	&=\ev_0(A+a\Phi)_{p(g,\theta)}((R_{(g,\theta)})_*X) + \ev_0(A+a\Phi)_{p(g,\theta)}(\iota_{p(g,\theta)}(\xi, x))\\
	&=\ad(g, \theta)^{-1}\ev_0(A+a\Phi)_{p} (X) + \ev_0(\xi + x \Phi(p(g,\theta)))\\
	&=\rho_{-\theta}\left( \ad(g^{-1})A(X)(0) + a(X) \ad(g^{-1})\Phi(p)(0) \right) \\
	&\phantom{=\rho_{-\theta}( \ad(g^{-1})A(X)(0) + a(X)}+ \xi + x\rho_{-\theta}\left( \ad(g^{-1}) \Phi(p)(0) \right)\\
	&= \ad(g^{-1})A(X)(\theta) + a(X) \ad(g^{-1})\Phi(p)(\theta) \\
	&\phantom{=\rho_{-\theta}( \ad(g^{-1})A(X)(0) + a(X)}+ \xi + x \ad(g^{-1}) \Phi(p)(\theta).
\end{align*}
That is,
$$
\widehat{\psi}^* \ev_0(A + a\Phi) = \ad(g^{-1}) A(\theta) + \Theta + \ad(g^{-1})\Phi(a + d\theta).
$$

\end{proof}


\section{The string class}\label{S:MS03}

\subsection{The lifting problem}
In this paper, we shall be primarily concerned with the so-called lifting problem for a principal bundle whose structure group has a central extension. (In particular, we will be concerned with the central extension of the semi-direct product $\LGS$.)

Suppose we are given a $K$-bundle $P \to M$ and a central extension 
\begin{equation}
\label{E:ce}
{U\vphantom{K}}(1) \to {\widehat{K}} \to {K}.
\end{equation}
 We would like to know when the bundle $P$ lifts to a $\widehat{K}$-bundle $\widehat{P}$.
If we denote by $\underline{H}$ the sheaf of smooth maps into a Lie group $H$ then the exact sequence 
\eqref{E:ce} gives rise to a corresponding short exact sequence of sheaves of groups and that in turn gives rise to a long  exact sequence in cohomology part of which is
$$
\cdots \to H^1(M, \underline{U\vphantom{K}}(1)) \to H^1(M, \underline{\widehat{K}}) \to H^1(M, \underline{K}) \to \cdots 
$$

Note, however, that since $K$ and $\widehat{K}$ are in general nonabelian, this is an exact sequence of \emph{pointed sets} rather than groups. Since $U(1)$ is central in $\widehat{K}$ we can extend this sequence one step to the right to obtain (see \cite{Brylinski:1993})
$$
H^1(M, \underline{U\vphantom{K}}(1)) \to H^1(M, \underline{\widehat{K}}) \to H^1(M, \underline{K}) \to H^2(M, \underline{U\vphantom{K}}(1)) 
$$
To define the coboundary map $H^1(M, \underline{K}) \to H^2(M, \underline{U\vphantom{K}}(1))$ take transition functions $g_{\a\b}$ for the $K$-bundle $P$ relative to some good cover $\{U_\a\}$ and lift them to maps $\hat{g}_{\a\b} \colon U_\a \cap U_\b \to \widehat{K}$ such that $p(\hat{g}_{\a\b}) = g_{\a\b}.$ These are our candidate transition functions for $\widehat{P}.$ However, transition functions are required to satisfy the cocycle condition $g_{\a\b}g_{\b\gamma} = g_{\a\gamma}$ on triple overlaps but the lifts $\hat{g}_{\a\b}$ only satisfy 
$$
\hat{g}_{\a\b} \hat{g}_{\b\gamma} = \epsilon_{\a\b\gamma}\hat{g}_{\a\gamma}
$$
for some $U(1)$-valued function $\epsilon_{\a\b\gamma}.$ It can be shown that the function $\epsilon_{\a\b\gamma}$ defines a cocycle in $H^2(M,\underline{U}(1))$ and this cocycle defines the image of $g_{\a\b}$ under the coboundary $H^1(M, \underline{K}) \to H^2(M, \underline{U\vphantom{K}}(1))$. We see that the $\hat{g}_{\a\b}$'s are transition functions precisely when this cocycle is trivial. Note that there is an isomorphism $H^2(M, \underline{U}(1)) \simeq H^3(M, \ZZ)$ induced by the exact sequence
$$
\ZZ \to \RR \to U(1).
$$
Therefore, under this isomorphism, we have that the obstruction to lifting $P$ to $\widehat{P}$ is a class in $H^3(M, \ZZ)$.

\subsection{The string class}

As a precursor to the problem we really want to consider---that of lifting a principal $\LGS$-bundle---we shall quickly describe the simpler case \cite{Murray:2003} of which our work is a generalisation.  Namely that of a \emph{string structure} for an $LG$-bundle. The loop group of a compact, simple, simply connected Lie group $G$ has a well-known central extension $\widehat{LG}$ (see \cite{Pressley-Segal} for details) and so the theory we have described in the previous section can be applied. In particular, given an $LG$-bundle $P \to M$ there is a class in $H^3(M),$ called the \emph{string class}, which represents the obstruction to lifting $P$ to an $\widehat{LG}$-bundle. If such a lifting exists (i.e.\! if the string class vanishes) then $P$ is said to have a string structure.  We shall always be interested in the image of the string class  in real cohomology which we shall  also call the  string class or occasionally the 
real string class.  Of course if $M$ has no three-dimensional torsion the string class and the real string class are equivalent information. 
String structures were first introduced by Killingback in \cite{Killingback:1987} as a string theory analogue of spin structures and studied further in \cite{Coquereaux:1989, Carey:1991, McLaughlin:1992, Murray:2003}. Our work is closely related to that in \cite{Murray:2003}. In that case,  the first author together with D.~Stevenson have given an explicit differential form based formula for the string class which we review. It involves a Higgs field for the bundle. A Higgs field for an $LG$-bundle is a map $\Phi \colon P \to L\fg$ which satisfies
$$
\Phi(p \gamma) = ad(\gamma^{-1}) \Phi (p) + \gamma^{-1} \partial \gamma,
$$
for all $\gamma \in LG$. They then prove

\begin{theorem}[\cite{Murray:2003}]\label{T:MS031}
Let $P \to M$ be a principal $LG$-bundle. Let $A$ be a connection on $P$ with curvature $F$ and let $\Phi$ be a Higgs field for $P.$ Then the real string class of $P$ is 
represented in de Rham cohomology by the form
$$
-\frac{1}{4\pi^2} \int_{S^1} \< F, \nabla \Phi\> \, d\theta,
$$
where $\<\, , \>$ is a suitably normalised invariant inner product on $\fg$ and
$$
\nabla \Phi = d\Phi + [A, \Phi] - \partial A.
$$
\end{theorem}

Additionally in \cite{Murray:2003}, a formula is given for the real string class in terms of the Pontrjagyn class of the caloron transform of  $P$. This formula generalises Killingback's original result which states that in the special case that $P \to M$ is a loop bundle (so $P = LQ$ and $M = LX$ for $Q \to X$ a $G$-bundle) then the  string class of $P$ is given by pulling-back the first Pontrjagyn class of $Q, p_1(Q),$ to $LX \times S^1$ and integrating over the circle:
$$
s(LQ) = \int_{S^1} \ev^*p_1(Q),
$$
where we have written $s(P)$ for the string class of $P$ and $\ev \colon LX \times S^1 \to X$ is the evaluation map which evaluates a loop at a point in $S^1.$ In order to generalise this result to general $LG$-bundles (which are not necessarily loop bundles) Murray and Stevenson use the caloron correspondence to relate the real string class to a class on a $G$-bundle. The generalisation of Killingback's result then is

\begin{theorem}[\cite{Murray:2003}]\label{T:MS032}
If $\widetilde{P} \to M\times S^1$ is the $G$-bundle corresponding to the $LG$-bundle $P \to M$ then the real string class of $P$ is given by
$$
s(P) = \int_{S^1} p_1(\widetilde{P}).
$$
\end{theorem}

\section{Central extensions}

\subsection{Preliminaries}
\label{sec:Ralpha_LG}
In this section we shall review some general results from \cite{Murray:2001, Murray:2003} on constructing central extensions of Lie groups before discussing the case where we have the circle acting and central extensions of loop groups. As in \eqref{E:ce} let $ \widehat K$
be a central extension of $K$ by $U(1)$.  Choose a connection $A$ for $\widehat K \to K$ thought of as $U(1)$-bundle and denote its curvature by 
$R \in \Omega^2(K)$. Let $\hat m \colon \widehat K \times \widehat K \to \widehat K$ be the multiplication and denote 
$$
\delta(A) = \pi_1^*(A) - m^*(A) + \pi_2^*(A).
$$
It is shown in \cite{Murray:2001, Murray:2003} that $\delta(A) = \pi^*(\a) $ for $\a \in \Omega^1(K \times K)$
and that moreover $\d(R) = d\a$ and $\d(\alpha) = 0$. Here we use $\delta(K^p) \to \delta(K^{p+1})$ to 
denote the alternating sum
$$
\d = \sum_{i=0}^p (-1)^i d_i^*.
$$
where $d_i \colon K^{p+1} \to K^{p}$ is defined by 
$$
d_i (k_1, \ldots, k_{p+1}) = \begin{cases}
					(k_2,\ldots, k_{p+1}), 	& i=0\\
					(k_1, \ldots, k_{i-1}k_i, k_{i+1}, \ldots, k_{p+1}), &1\leq i \leq p-1\\
					(k_1, \ldots, k_p),		& i=p
					\end{cases}
$$

Conversely choose a pair of forms $(R, \a)$ such that $R \in  \Omega^2(K)$, $\a \in \Omega^2(K \times K)$, $dR = 0$, $2\pi i R $ is integral, 
$\delta(R) = d\a$ and $\delta(\a) = 0$.  Then we can explicitly construct a central extension $\widehat{K}\to K$ with connection from 
which we can recover $(R, \a)$ by the above construction.  Two pairs $(R, \a)$ and $(R', \a')$ give rise to the same central 
extension if and only if there  is a one-form $\eta \in \Omega^1(K)$ such that $R = R' + d\eta$ and $\a = \a' + \d \eta$. 

\begin{example}
\label{ex:Ralpha}
In the case that  $K$ is the loop group $LG$ the central extension whose class is the generator  
of $H^2(LG, \ZZ)$ is shown in  \cite{Murray:2003} to be determined by 
\begin{align}\label{E:R}
R &= \frac{i}{4\pi} \int_{S^1} \< \Theta, \partial \Theta \>\, d\theta,\\
\a &= \frac{i}{2\pi} \int_{S^1} \< d_2^* \Theta, d_0^*Z \> \, d\theta\label{E:alpha}
\end{align}
for $\Theta$ the Maurer-Cartan form on $G$. The bracket here is an invariant inner product on $\fg$ normalised so that the longest root has length squared equal to 2 and $Z$ is the function on $LG; \, \gamma \mapsto \partial \gamma \gamma^{-1}.$
\end{example}

\subsection{Circle actions and semi-direct products}

We specialise now to the case of $LG \rtimes_\rho S^1$ for the usual homomorphism $\rho \colon S^1 \to \Aut(LG)$. 
Let $\widehat{LG} \to LG$ be the standard central extension whose Chern class is a generator of $H^2(LG, \ZZ)$. 
We will show that there is a unique central extension of $LG \rtimes_\rho S^1$ whose Chern class is the class
pulled back to $H^2(LG \rtimes_\rho S^1, \ZZ)$.

For existence note that  $\widehat{LG} \to LG$ can be constructed explicitly \cite{Murray:2003} using the $(R, \a)$ in Example \ref{ex:Ralpha} and, moreover, these $(R, \a)$ are invariant under the action given by $\rho$. It follows that we can 
lift the action $\rho$ to an action $\hat\rho \colon S^1 \to \Aut(\widehat{LG})$ and then $\widehat{LG} \rtimes_{\hat\rho} S^1$ is the required central extension of $LG \rtimes_\rho S^1$.

For uniqueness we will show that any central extension $\cH \to K \rtimes S^1$
is determined by its Chern class as a $U(1)$-bundle.  Consider first the corresponding problem for the loop group $LG$.
The observant reader will notice that in \cite{Murray:2003} the central extension of $LG$ was defined
by choosing an $(R, \a)$ rather than taking one of the standard constructions and calculating $(R, \a)$.
The question, of course, is to find $\a$ as $R$ is essentially defined by specifying the Chern class of the 
central extension.  While still avoiding the actual calculation we can resolve the question of whether
we had the correct $(R, \a)$ in \cite{Murray:2003} from the following. 

\begin{proposition} If $\cK \to LG$ and $\cH \to LG$ are central extensions which have the same
Chern class in $H^2(LG, \ZZ)$ then they are isomorphic as central extensions. \end{proposition}

\begin{proof} We can assume that  $\cK \to LG$ and $\cH \to LG$ are the same $U(1)$-bundle with possibly
different multiplications so let us denote that by $\cK \to LG$. 
By choosing a connection on the bundle we can construct $(R, \a)$ and $(R, \b)$ characterising the two 
central extensions. We have $\delta(R) = d\a = d\b$ so that $d(\a - \b) = 0$. As $LG$ is simply-connected there is a $\chi \colon LG \times LG \to S^1$ such that $ \a = \b + \chi^{-1} d\chi$. Without loss of generality we can normalise so that $\chi(e, e) = 1$.  As $\delta(a) = \delta(b) = 0$ we have 
$\delta(\chi^{-1})d \delta(\chi) = 0 $ or $\delta(\chi)$ is a constant.  Evaluating at $(e, e, e)$ we find that 
$\delta(\chi) = 1$ or 
$$
\chi(g_1, g_2) = \chi(g_2, e)^{-1} \chi(g_1 , e)^{-1}
$$
so that letting $\mu(g) = \chi(g, e)^{-1}$ we see that $\delta(\mu) = \chi$. If we let 
$\eta = \mu^{-1} d\mu$ we now have $R = R + d\mu$ and $\a = \b + \delta \eta$ so the 
the central extensions are isomorphic. 
\end{proof}

 Notice that any automorphism of $\widehat{LG}$ fixes the centre so descends to an 
automorphism of $LG$ so that we have  a homomorphism $\Aut(\widehat{LG}) \to \Aut(LG)$.  Let 
 $\cH \to LG \rtimes_\rho S^1$ be a central extension whose Chern class is the pull-back of the Chern class
 of $\widehat{LG} \to LG$.  The result above shows that the restriction of  $\cH$ to $LG$ 
 is isomorphic, as a central extension, to $\widehat{LG} \to LG$. Hence we have  a short exact sequence of
 central extensions of the form:

\begin{equation*} 
\label{eq:fibre-integration} 
\xy 
(-55,7.5)*+{\widehat{LG}}="1"; 
(-35,7.5)*+{\cH}="2"; 
(-15,7.5)*+{S^1}="3"; 
(-55,-7.5)*+{LG}="4"; 
(-35,-7.5)*+{LG \rtimes_\rho S^1}="5"; 
(-15,-7.5)*+{S^1}="6"; 
{\ar "1";"2"};
{\ar "2";"3"};
{\ar "4";"5"};
{\ar "5";"6"};
{\ar "1";"4"};
{\ar "2";"5"};
{\ar "3";"6"}
\endxy
\end{equation*}

The map $z \mapsto (1, z)$ defines an homomorphism $S^1 \to LG \rtimes_\rho S^1$ and pulling 
back $\cH \to LG \rtimes_\rho S^1$ defines a central extension of the circle by the circle. This 
is trivial by

\begin{proposition} 
Any central extension of the circle by the circle is trivial.
\end{proposition} 

\begin{proof} 
For convenience let us denote the central extension by 
$$ 
 U(1) \to H \to S^1 
$$
where $U(1) \subset H$ is in  the centre of $H$. We prove the existence of a splitting homomorphism $S^1 \to H$. 

First choose a right $H$ invariant connection on $H \to S^1$ by translating around some splitting
of the tangent space at $e$.  Consider the homomorphism $\RR \to S^1$ and lift it to   a horizontal map
$f \colon  \RR \to H$ with $f(0) = e$.   Uniqueness for ordinary differential equations and the 
fact that $ t\mapsto f(t + s)f(s)^{-1}$ is a horizontal lift as well shows that $f$ is a homomorphism. 

 We would like $f \colon \RR \to H$ to descend to $ S^1 \to  H$ but generally we will have $f(2\pi) \neq e$.   Choose 
$h : \RR \to  U(1)$ a homomorphism with $h(2\pi) = f(2\pi)^{-1}$.  Then, because $U(1)$ is central
$hf$  is a homomomorphism which gives the required splitting.
\end{proof}

It follows that we have a homomorphism $\sigma \colon S^1 \to \cH$ and an induced $\hat\rho \colon S^1 \to \Aut(\widehat{LG})$ making $\cH$ a semi-direct product $\widehat{LG} \rtimes_{\hat\rho} S^1$.  Uniqueness follows from the uniqueness of $\hat \rho$ which, in turn, follows from:

\begin{proposition}
 The homomorphism $\Aut(\widehat{LG}) \to \Aut(LG)$ is injective.
 \end{proposition}
 
 \begin{proof} Assume that $\phi \in \Aut(\widehat{LG})$ acts trivially on $LG$. It follows that 
 for any $\hat g \in \widehat{LG}$ we have $\phi(\hat g) = \hat g \chi(g)$ for some homomorphism
 $\chi \colon \widehat{LG} \to U(1)$. If we restrict the homomorphism $\chi$ to $U(1) \subset \widehat{LG}$
 then it must take the form $\chi(z) = z^p$ for some $p \in \ZZ$. If $p=0$ we are done. Otherwise consider the 
 kernel $K \subset \widehat{LG}$ of $\chi$. From \cite{Pressley-Segal} we know that $LG$ is equal to its own 
 commutator so that $ [\widehat{LG}, \widehat{LG}] \subset \widehat{LG}$. But 
 $ [\widehat{LG}, \widehat{LG}] \subset K$ so that  $K$  covers $LG$. Thus $K$ is a reduction of $\widehat{LG}$
 to $\ZZ_p$ and the Chern class of $\widehat{LG}$ is torsion which is a contradiction. So $K = \{e\}$. 
\end{proof}

\subsection{Central extension of loop group and semi-direct product}
In what follows we will require an explicit construction of the central extension of $LG\rtimes_\rho S^1$ in a manner
analogous to that in Section \ref{sec:Ralpha_LG}.  However that  construction needs a slight modification because
$LG\rtimes_\rho S^1$ is not simply-connected.   This is done \cite{{Murray:2003}} by  replacing the two-form $R$ with a differential character \cite{Cheeger:1985} for the bundle $\widehat{K} \to K$. That is, we add to our pair $(R, \a)$ a homomorphism $h\colon Z_1(K) \to U(1)$ satisfying
$$
h(\partial \sigma) = \exp \left( \int_\sigma R \right)
$$
for every two-cycle $\sigma$ in $K.$ We also require the compatibility condition
$$
(\d h) (\gamma) = \exp \left( \int_\gamma \a \right)
$$
for every closed one-cycle $\gamma$ in $K \times K.$
Therefore, we need to find a triple of objects $(R, \a, h)$ as above. 

Following from the result in the previous section we take 
as our $R$  the pull-back of the form (\ref{E:R}) above to $\LGS.$ That is,
$$
R = \frac{i}{4\pi} \int_{S^1} \< \Theta, \partial \Theta \>\, d\theta.
$$
As noted before, since we are integrating over the circle this expression is invariant under the action $\rho$. Now, to find $\a$ we need to calculate $\d R = \pi_1^*R - m^*R + \pi_2^*R,$ where $\pi_i$ is the projection $\LGS \times \LGS \to \LGS$ which omits the $i^\text{th}$ factor and $m$ is the multiplication defined above. The pull-back $\pi_i^*R$ is given by
$$
\frac{i}{4\pi} \int_{S^1} \< \pi_i^*\Theta, \partial \pi_i^*\Theta \>\, d\theta
$$
and so it remains to calculate $m^*R.$ From now on we shall write $\Theta_1$ for $\pi_2^* \Theta$ and so on. Using the fact that $R$ is $\rho$-invariant as well as the $\ad$-invariance of the inner product and the identity
\begin{equation}\label{E:d(ad)}
\partial \left(\ad(\gamma^{-1}) X \right) = \ad(\gamma^{-1})[X, Z] + \ad(\gamma^{-1}) \partial X,
\end{equation}
for a tangent vector $X$ and $Z$ the function on $LG$ defined above, we have
\begin{multline*}
m^*R = \frac{i}{4\pi} \int_{S^1} \left\< [\Theta_1, \Theta_1], Z_2 \right\>
+ \left\< \Theta_1, \partial \Theta_1 \right\>
+ \left\< \Theta_2, \partial \Theta_2 \right\>\\
+ 2 \left\< \ad(\gamma_2^{-1}) \Theta_1, \partial \Theta_2^{\rho_2} \right\>
- 2\left\< \mu_1 \ad(\gamma_2^{-1}) Z_2, \partial \Theta_2^{\rho_2} \right\>
- 2\left\< \mu_1 Z_2, \partial \Theta_1 \right\> d\theta,
\end{multline*}
where $\mu$ is the Maurer-Cartan form on $S^1$ and we have written (for example) $\xi_1^{\rho_1}$ for $\rho_{\phi_1}(\xi_1)$. Therefore
\begin{multline*}
\d R = \frac{i}{2\pi} \int_{S^1} -\tfrac{1}{2} \left\< [\Theta_1, \Theta_1], Z_2 \right\>
- \left\< \ad(\gamma_2^{-1}) \Theta_1, \partial \Theta_2^{\rho_2} \right\>\\
+ \left\< \mu_1 \ad(\gamma_2^{-1}) Z_2, \partial \Theta_2^{\rho_2} \right\>
+ \left\< \mu_1 Z_2, \partial \Theta_1 \right\> d\theta.
\end{multline*}
Now define
\begin{equation}\label{E:alphaLGS}
\a = \frac{i}{2\pi} \int_{S^1} \big\< \pi_2^*\Theta^{\rho^{-1}} \! - \tfrac{1}{2} \pi_2^*\mu \, \pi_1^*Z, \pi_1^*Z \big\> \, d\theta.
\end{equation}
Then it is easy to check that $d\a = \d R$ and $\d\a = 0.$ Notice that the $2$-form $R$ is left invariant and the $1$-form $\a$ is left invariant in the first slot. To find the homomorphism $h\colon Z_1(\LGS) \to U(1)$ we note that since $\pi_1(\LGS) = \ZZ$ any cycle $a \in Z_1(\LGS)$ can be written as $n\gamma + \partial \sigma,$ for some two-cycle $\sigma,$ where $\gamma$ is the generator of $H_1(\LGS),$ a loop around the $S^1$ factor. It is easy to see that the integral of $\a$ over the generators of $H_1(\LGS \times \LGS)$ vanishes, that is,
$$
\int_{\gamma_1} \a = 0 = \int_{\gamma_2} \a
$$
for $\gamma_1, \gamma_2$ loops around the first and second $S^1$ factors respectively. This suggests that we define
\begin{equation}\label{E:h}
h(a) = h(\partial \sigma) = \exp \left( \int_\sigma R \right).
\end{equation}
This is well defined since if $a = n\gamma + \partial \sigma = n \gamma + \partial \sigma'$ then $\partial (\sigma - \sigma') = 0$ and so $\ds\int_{\sigma - \sigma'} R \in 2\pi i \ZZ$ (since $R$ is integral). Because the integral of $\a$ over the generators of $H_1(\LGS \times \LGS)$ vanishes, it is easy to check that for any one-cycle $\gamma$ we have
$$
(\d h) (\gamma) = \exp \left( \int_\gamma \a \right).
$$
Thus we have proven
\begin{proposition}
The triple $(R, \a, h)$ as above determines the central extension of the semi-direct product $\LGS.$
\end{proposition}

\section{Lifting bundle gerbes}
\label{sec:liftingbg}

In order to perform calculations with the differential forms in the previous section we shall utilise the theory of \emph{bundle gerbes}, introduced in \cite{Murray:1996}. 
Let $Y \xrightarrow{\pi} M$ be a surjective submersion.
\begin{definition}[\cite{Murray:1996}]
A \emph{bundle gerbe} over a manifold $M$ is a pair $(P, Y)$ where $Y \to M$ is a surjective submersion and $P \to Y^{[2]}$ is a $U(1)$-bundle and such that there is a \emph{bundle gerbe multiplication}, which is a smooth isomorphism
$$
m \colon \pi_3^* P \otimes \pi_1^* P \xrightarrow{\sim} \pi_2^* P
$$
of $U(1)$-bundles over $Y^{[3]}.$ Further, this multiplication is required to be associative whenever triple products are defined. We sometimes denote a bundle gerbe simply by $P.$
\end{definition}

We can characterise the bundle gerbe multiplication and its associativity in a different way using sections of bundles related to $P$ as follows. If $Q\to Y^{[p-1]}$ is a $U(1)$-bundle, define the bundle $\d Q \to Y^{[p]}$ as
$$
\d Q = \pi_1^*Q \otimes (\pi_2^* Q )^* \otimes \pi_3^* Q \otimes \ldots
$$
Then it is easy to show that $\d\d Q$ is canonically trivial. One can show that the bundle gerbe multiplication is equivalent to a section $s$ of $\d P \to Y^{[3]}$ and that the associativity condition is equivalent to the condition that $\d s = 1$ as a section of $\d\d P$ (where $1$ denotes the canonical section of $\d\d P$). Indeed if $p$ and $q$ are elements of $P_{(y_1, y_2)}$ and $P_{(y_2, y_3)}$ respectively, we can define a section $s$ of $\d P$ by
$$
s(y_1, y_2, y_3) = p\otimes m(p,q)^*\otimes q,
$$
then the associativity of $m$ forces the condition $\d s =1.$ 

If $\widehat{K}$ is a central extension of $K$ as before, then the specific bundle gerbe we are interested in is constructed as follows. Take the principal $K$-bundle $P \to M$ and consider the fibre product $P^{[2]} \rightrightarrows P$ then there is a natural map $\tau \colon P^{[2]} \to K,$ called the \emph{difference map}, given by $p_1 \tau(p_1, p_2) = p_2.$ If we view $\widehat{K}$ as a $U(1)$-bundle over $K$ then we can pull-back $\widehat{K}$ by this map to obtain a $U(1)$-bundle over $P^{[2]}:$
$$
\xymatrix{ \tau^* \widehat{K} \ar[r] \ar[d]	& \widehat{K} \ar[d]\\
			P^{[2]} \ar[r]^\tau			& K^{\vphantom{[2]}}}
$$
where
$$
\tau^* \widehat{K} = \left\{ (p_1, p_2, \hat{g}) \mid p(\hat{g}) = \tau(p_1, p_2) \right\}.
$$
Note that $\tau(p_1, p_2) \tau(p_2, p_3) = \tau(p_1, p_3)$ and so, because the multiplication in $\widehat{K}$ covers that in $K,$ we have an induced map
$$
\tau^*\widehat{K}_{(p_1, p_2)} \otimes \tau^*\widehat{K}_{(p_2, p_3)} \to \tau^*\widehat{K}_{(p_1, p_3)}
$$
which serves as a bundle gerbe multiplication for the bundle gerbe $(\tau^*\widehat{K}, P)$ over $M.$ This bundle gerbe is called the \emph{lifting bundle gerbe}.

Bundle gerbes have a characteristic class associated with them called the Dixmier-Douady class which in the case of the lifting bundle gerbe  is precisely the obstruction to lifting $P$ to a $\widehat{K}$-bundle.

One can write down a differential form representative of (the image in real cohomology of) the Dixmier-Douady class of a bundle gerbe in the same way that the curvature can be used to represent the Chern class of a $U(1)$-bundle. This requires a \emph{connection} and \emph{curving} for the bundle gerbe.

Consider first the $p$-fold fibred product $Y^{[p]}$ as before. Let $\Omega^q(Y^{[p]})$ denote the space of differential $q$-forms on $Y^{[p]}.$ Then we can define a map $\d : \Omega^q(Y^{[p]}) \to \Omega^q(Y^{[p+1]})$ as the alternating sum of pull-backs by the projections $\pi_i: Y^{[p+1]} \to Y^{[p]}$ which omit the $i$th element:
$$
\d = \sum_{i=1}^{p+1} (-1)^{i-1} \pi_i^*.
$$
Then $\d^2 = 0$ and so we have a complex
$$
0 \to \Omega^q(M) \xrightarrow{\pi^*} \Omega^q(Y) \xrightarrow{\,\d\,} \Omega^q(Y^{[2]}) \xrightarrow{\,\d\,} \Omega^q(Y^{[3]}) \xrightarrow{\,\d\,} \ldots
$$
In \cite{Murray:1996} it is proven that this complex has no cohomology. That is, the above sequence is exact for all $q \geq 0.$ We shall use this result shortly.

A bundle gerbe connection is a connection $A$ for the $U(1)$-bundle $P$ that respects the bundle gerbe product in the sense that the induced connection on $\pi_2^* P$ is the same as the image of the induced connection on $\pi_3^*P \otimes \pi_1^*P$ under the bundle gerbe multiplication. Equivalently using the section $s$ we have $s^*(\delta A) = 0$. 
If $F$ is the curvature of a bundle gerbe connection $A$ viewed as a 2-form on $Y^{[2]},$ then $\d F = s^* (\d dA) = d( s^* (\d A)) = 0.$ This means that there is some $B \in \Omega^2(Y)$ satisfying $F = \d B.$ A choice of such a $B$ is called a \emph{curving} for $P.$ Note that if $B'$ is another choice of curving then $B$ and $B'$ differ by a $\d$-closed (and hence $\d$-exact) 2-form on $Y$. As $\d$ and $d$ commute, we have that $\d (dB) = d(\d B) = dF = 0.$ Therefore there is a 3-form $H$ on $M$ such that $dB = \pi^* H$ (for $\pi$ the projection $Y \to M$). $H$ is called the \emph{3-curvature} of $P.$ It is closed and a different choice of $B$ or $H$ would result in a difference of an exact form. So $H$ defines a cohomology class in $H^3(M).$ It turns out that the 3-form $H/2\pi i$ is integral and that $H/ 2\pi i$ is a representative of the image of the Dixmier-Douady class of $P$ in real cohomology.

\section{The real string class of an $LG \rtimes_\rho S^1$-bundle}

\subsection{The real string class of an $LG \rtimes_\rho S^1$-bundle}
In the section \ref{S:MS03} we outlined the results from \cite{Murray:2003} which give a formula for the real string class of an $LG$-bundle and use the fact that $LG$-bundles over $M$ correspond to $G$-bundles over $M\times S^1.$ We have also showed that if one extends this to non-trivial $S^1$-bundles over $M$ one obtains an $\LGS$-bundle over $M$. In this section we present a formula for the image in real cohomology of the obstruction to lifting a principal $\LGS$-bundle $P$ to a principal $\widehat{LG} \rtimes_{\hat\rho} S^1$-bundle $\widehat{P},$ which we call the \emph{real string class} of $P.$

\subsubsection{A connection for the lifting bundle gerbe}
Now that we have a construction of the central extension of $\LGS$ in terms of the differential forms $R$ and $\a,$ we can consider the problem of lifting the $\LGS$-bundle $P \to M$ to an $\widehat{LG} \rtimes_{\hat\rho} S^1$-bundle $\widehat{P} \to M.$ We can write down the lifting bundle gerbe for this problem, that is, the bundle gerbe $(\tau^*(\widehat{LG} \rtimes_{\hat\rho} S^1), P)$ over $M,$ and we would like a connection on this bundle gerbe so we can calculate its Dixmier-Douady class. 

Consider, then, the difference map $\tau \colon P^{[2]} \to \LGS$ as in Section \ref{sec:liftingbg}. We can extend this to a map $\tau \colon P^{[k+1]} \to (\LGS)^k$ by defining
$$
\tau (p_1, \ldots, p_{k+1}) = (\tau(p_1, p_2), \ldots, \tau(p_{k}, p_{k+1})).
$$
It is shown in \cite{Murray:2003} that this kind of map is a simplicial map and thus we have commuting diagrams
 
\begin{equation*} 
\label{eq:tau} 
\xy 
(-65,7.5)*+{\Omega(P^{[p]})}="1"; 
(-30,7.5)*+{\Omega(P^{[p+1]})}="2";  
(-65,-7.5)*+{(\LGS)^{p-1}}="3"; 
(-30,-7.5)*+{(\LGS)^{p}}="4"; 
{\ar^\delta "1";"2"};
{\ar^\delta "3";"4"};
{\ar_\tau "1";"3"};
{\ar_\tau "2";"4"}
\endxy
\end{equation*}

Now consider a connection $\nu$ on $\widehat{LG} \rtimes_{\hat\rho} S^1$ (whose curvature is the form $R$). The natural choice for a bundle gerbe connection would be the pull-back, $\tau^*\nu,$ of this form to $\tau^*(\widehat{LG} \rtimes_{\hat\rho} S^1).$ However, $\tau^*\nu$ is not a bundle gerbe connection because it does not respect the product. That is, $s^*(\d\tau^*\nu)$ is non-zero. However, $\d(s^*(\d\tau^*\nu)) = 0$ and so there is some form $\epsilon$ on $P^{[2]}$ such that $\d \epsilon = s^*(\d\tau^*\nu).$ Then $\tau^*\nu - \epsilon$ will be a bundle gerbe connection on $\tau^*(\widehat{LG} \rtimes_{\hat\rho} S^1).$ In fact, in this case, since $\a = s^*(\d\nu)$ by definition, we have $s^*(\d \tau^*\nu) = \tau^*\a.$ So $\d(s^*(\d \tau^*\nu)) = \d \tau^*\a = \tau^* \d\a = 0$ as $\d\a = 0$ and so $\epsilon$ satisfies $\d\epsilon = \tau^* \a.$ Thus it suffices to find a 1-form $\epsilon$ on $P^{[2]}$ satisfying $\d\epsilon = \tau^* \a.$ In fact, it is possible to write $\epsilon$ in general in terms of $\a$ \cite{Stevenson:comm}. We shall now demonstrate how to do this. Let $P$ be a $K$-bundle with connection $A.$ Using the equation $p_1 \tau(p_1, p_2) = p_2$ and the Leibnitz rule, we find the identity
\begin{equation}\label{E:tau and A}
\pi_1^*A = \ad(\tau_{12}^{-1})\pi_2^* A + \tau_{12}^* \Theta,
\end{equation}
where we have written $\tau_{ij}$ for $\tau(p_i, p_j).$ For tangent vectors $(X_1, X_2, X_3)$ at $(p_1, p_2, p_3) \in P^{[3]},$ we can calculate
\begin{multline*}
(\d\a)_{(1, \tau_{12}, \tau_{23})}( A(X_1), \tau_{12}(X_1, X_2), \tau_{23}(X_2,X_3)) =\\
	\shoveleft{\phantom{(\d\a)}\a_{(\tau_{12},\tau_{23})}( \tau_{12}(X_1, X_2), \tau_{23}(X_2,X_3))}\\
	 - \a_{(\tau_{12},\tau_{23})}( m_*(A(X_1),\tau_{12}(X_1, X_2)), \tau_{23}(X_2,X_3))\\
	+ \a_{(1,\tau_{12}\tau_{23})}( A(X_1),m_*(\tau_{12}(X_1, X_2), \tau_{23}(X_2,X_3)))\\
	 - \a_{(1,\tau_{12})}(A(X_1),\tau_{12}(X_1, X_2)).
\end{multline*}
Notice that the first term above is actually $\tau^* \a.$ Since $\d\a = 0,$ we have
\begin{multline*}
(\tau^*\a)_{(p_1,p_2,p_3)} (X_1,X_2,X_3) =\\
	\a_{(\tau_{12},\tau_{23})}(m_*(A(X_1),\tau_{12}(X_1, X_2)), \tau_{23}(X_2,X_3))\\
	- \a_{(1,\tau_{12}\tau_{23})}(A(X_1),m_*(\tau_{12}(X_1, X_2), \tau_{23}(X_2,X_3)))\\
	+ \a_{(1,\tau_{12})}(A(X_1),\tau_{12}(X_1, X_2)).
\end{multline*}
Now, if we define $\epsilon$ in terms of $\a$ and $A$ as
$$
\epsilon_{(p_1,p_2)}(X_1, X_2) =\a_{(1, \tau_{12})}(A(X_1), \tau_{12}(X_1,X_2))
$$
then we have
\begin{multline*}
(\d\epsilon)_{(p_1, p_2,p_3)}(X_1,X_2,X_3) =\\
	\a_{(1, \tau_{23})}(A(X_2), \tau_{23}(X_2, X_3)) - \a_{(1, \tau_{13})}(A(X_1), \tau_{13}(X_1,X_3))\\
	+\a_{(1,\tau_{12})}(A(X_1), \tau_{12}(X_1,X_2)).
\end{multline*}
Using the fact that $\tau_{13} = \tau_{12}\tau_{23},$ we see
$$
\a_{(1, \tau_{13})}( A(X_1), \tau_{13}(X_1,X_2))=
\a_{(1,\tau_{12}\tau_{23})}( A(X_1),m_*(\tau_{12}(X_1, X_2), \tau_{23}(X_2,X_3)))
$$
and since $\a$ is left invariant in the first slot, and using equation (\ref{E:tau and A}) we have
\begin{align*}
\a_{(1, \tau_{23})}( & A(X_2), \tau_{23}(X_2, X_3))\\
			&= \a_{(\tau_{12}, \tau_{23})}( \tau_{12}A(X_2), \tau_{23}(X_2, X_3))\\
			&= \a_{(\tau_{12}, \tau_{23})}( \tau_{12}\ad(\tau_{12}^{-1})A(X_1) + \tau_{12}(X_1,X_2), \tau_{23}(X_2, X_3)),
\end{align*}
which equals
$$
\a_{(\tau_{12},\tau_{23})}( m_*(A(X_1),\tau_{12}(X_1, X_2)), \tau_{23}(X_2,X_3)).
$$
Thus we have $\d \epsilon = \tau^*\a.$

Consider now the $\LGS$-bundle $P.$ Choose a connection $(A, a)$ for $P,$ where $A$ and $a$ are $1$-forms on $P$ with values in $L\fg$ and $\RR$ respectively (as per Proposition \ref{P:(A,a)}). 
Given $(A, a)$ then, we can write down the $1$-form $\epsilon \in \Omega^1(P^{[2]})$ as above:
$$
\epsilon = \frac{i}{2\pi} \int_{S^1} \left\< \pi_2^* A - \tfrac{1}{2} \pi_2^*a \, \tau^*Z, \tau^*Z \right\> d\theta.
$$
It is easy to check that $\d \epsilon = \tau^* \a$ and so we have that $\tau^*\nu - \epsilon$ is a connection for the lifting bundle gerbe. Of course, we are concerned with finding a curving for this bundle gerbe and so we are really interested in calculating the curvature of this connection, given by $\tau^*R - d\epsilon.$ For the connection $(A, a),$ equation (\ref{E:tau and A}) reads
$$
(A_2,a_2) =\left( \rho_{-\tau_{S^1}}\left(\ad(\tau_{LG}^{-1})A_1 - a_1 \tau_{LG}^{-1} \partial\tau_{LG}^{\vphantom{-1}} \right) + \tau_{LG}^* (\rho_{-\tau_{S^1}}(\Theta) ), a_1 + \tau_{S^1}^*\mu \right)
$$
where we have written the difference map $\tau$ as $(\tau_{LG}, \tau_{S^1}).$ That is, $\tau_{LG}$ is the $LG$ part of $\tau$ and $\tau_{S^1}$ is the circle part. From now on, we will simply write $\tau$ and assume that it is clear from the context which part we mean. In particular, then, we have
\begin{equation}\label{E:tau and (A,a)}
\tau^* \rho_{-\tau}(\Theta) =A_2 - \rho_{-\tau}\left(\ad(\tau^{-1})A_1 + a_1\tau^{-1} \partial\tau \right).
\end{equation}
Note that here we have used the fact that the Maurer-Cartan form on $\LGS$ is not the pair $(\Theta, \mu)$ but in fact includes a rotation of $\Theta.$ So at the point $(\gamma, \phi),$ it is given by $(\rho_{\phi^{-1}}(\Theta), \mu).$ Using equation (\ref{E:tau and (A,a)}) and writing $A^\rho$ for $\rho(A)$ and so on as before, we calculate
\begin{multline*}
\tau^*R = \frac{i}{4\pi} \int_{S^1} \<A_2, \partial A_2 \> - 2\<A_2^\rho , \partial  (\ad(\tau^{-1})A_1) \> + 2\<A_2^\rho , a_1\partial (\tau^{-1} \partial\tau)\>\\
+ \< \ad(\tau^{-1})A_1 , \partial  (\ad(\tau^{-1})A_1)\> - 2\< \ad(\tau^{-1})A_1 , a_1 \partial  (\tau^{-1} \partial\tau)\> d\theta.
\end{multline*}
For $d\epsilon$ we have:
\begin{multline*}
d\epsilon = \frac{i}{2\pi} d \int_{S^1}\< A_1 - \tfrac{1}{2} a_1 \tau^*Z, \tau^*Z \>d\theta\\
\shoveleft{\phantom{d\epsilon} = \frac{i}{2\pi} \int_{S^1} \< dA_1, \tau^*Z\>
 - \< A_1, d(\tau^*Z)\> 
 - \tfrac{1}{2}\< da_1\tau^*Z, \tau^*Z\>
 + \< a_1 \tau^*Z, d(\tau^*Z)\> d\theta}\\
\end{multline*}
which, using the identity
\begin{equation}
d(\tau^*Z) = \ad(\tau)\partial (\tau^*\Theta^\rho)
\end{equation}
and equation (\ref{E:tau and (A,a)}), gives
\begin{multline*}
d\epsilon = \frac{i}{2\pi} \int_{S^1} \< dA_1, \tau^*Z\>
 - \< A_1, \ad(\tau) \partial A_2^\rho \>
  + \< A_1, \ad(\tau) \partial (\ad(\tau^{-1})A_1) \>\\
- \< A_1, a_1 \ad(\tau) \partial (\tau^{-1} \partial\tau)\>
 - \tfrac{1}{2} \<da_1\tau^*Z, \tau^*Z\>\\
+ \< a_1 \tau^*Z, \ad(\tau)\partial A_2^\rho \>
 - \< a_1 \tau^*Z, \ad(\tau)\partial (\ad(\tau^{-1})A_1) \> d\theta.
\end{multline*}
Therefore,
\begin{multline*}
\tau^*R - d\epsilon = \frac{i}{4\pi} \int_{S^1} \<A_2, \partial A_2 \>
-  2\< dA_1, \tau^*Z\>
- \< A_1, \ad(\tau)\partial (\ad(\tau^{-1})A_1) \>\\
+ 2\< a_1 \tau^{-1} \partial\tau, \partial(\ad(\tau^{-1})A_1)\>
+ \<da_1\tau^*Z, \tau^*Z\> d\theta,
\end{multline*}
using the $\ad$ invariance of the inner product and integration by parts. Then, using equation (\ref{E:d(ad)}) yields
\begin{multline*}
\tau^*R - d\epsilon = \frac{i}{4\pi} \int_{S^1} \<A_2, \partial A_2 \> 
- \< A_1 , \partial A_1\> 
-  2\< dA_1, \tau^*Z\> 
- \< [A_1, A_1],\tau^*Z \> \\
+ 2 \< \tau^*Z a_1, \partial A_1\> 
+ \<da_1\tau^*Z, \tau^*Z\> d\theta.
\end{multline*}
Note now that if $(F,f)$ is the curvature of the connection $(A,a)$ then we have
$$
(F,f) = (dA + \tfrac{1}{2} [A,A] - a\wedge \partial A, da).
$$
Therefore, the formula above for $\tau^*R - d\epsilon$ reads
$$
\tau^*R - d\epsilon = \frac{i}{4\pi} \int_{S^1} \left\<\pi_1^*A, \partial \pi_1^*A \right\> - \left\< \pi_2^*A , \partial \pi_2^*A \right\> - 2\left\< \pi_2^* F - \tfrac{1}{2} \pi_2^* f \, \tau^*Z, \tau^*Z\right\> d\theta.
$$

\subsubsection{A curving for the lifting bundle gerbe}

Recall that in order to find the 3-curvature of the lifting bundle gerbe, and hence a representative for the image in real cohomology of the Dixmier-Douady class, we need a curving for $\tau^*(\widehat{LG} \rtimes_{\hat\rho} S^1).$ That is, some 2-form $B$ on $P$ such that $\d B = \tau^* R - d\epsilon.$ Note that $\d = \pi_1^* - \pi_2^*$ and
\begin{equation}\label{tau*R - de}
\tau^*R - d\epsilon = \d \left(\frac{i}{4\pi} \int_{S^1} \left\< A, \partial A \right\> d\theta \right)
- \frac{i}{2\pi} \int_{S^1} \left\< \pi_2^* F - \tfrac{1}{2} \pi_2^* f \, \tau^*Z, \tau^*Z\right\> d\theta.
\end{equation}
To deal with the second term above, we need a Higgs field for the $\LGS$-bundle $P$.  Recall from Definition \ref{def:Higgsfield} that Higgs fields satisfy
$$
\Phi(p (\gamma, \phi)) = \rho_{-\phi} \left(\ad(\gamma^{-1}) \Phi(p) + \gamma^{-1} \partial \gamma \right).
$$

Note that this condition implies that
$$
\pi_1^*\Phi = \rho_{-\tau} \left( \ad(\tau^{-1}) \pi_2^* \Phi + \tau^{-1} \partial \tau \right)
$$
or simply,
\begin{equation}\label{E:Phi and tau}
\ad(\tau) \Phi_2^\rho = \Phi_1 + \tau^*Z.
\end{equation}
Using this, the second term in equation (\ref{tau*R - de}) becomes
$$
\frac{i}{2\pi} \int_{S^1} \left\< F_1 - \tfrac{1}{2} f_1 \, \tau^*Z, \ad(\tau) \Phi_2^\rho - \Phi_1 \right\> d\theta.
$$
Since $(F,f)$ is a curvature, it satisfies
$$
\pi_1^* (F,f) = \ad(\tau^{-1}) \pi_2^*(F,f).
$$
That is, $f_2 = f_1$ and
$$
F_2 = \rho_{-\tau} \left(\ad(\tau^{-1}) F_1 - f_1 \tau^{-1}\partial \tau \right),
$$
or
$$
\ad(\tau) F_2^\rho = F_1 -  f_1 \tau^*Z.
$$
Using this, we have
\begin{align*}
\frac{i}{2\pi} \int_{S^1} 	&\left\< F_1 - \tfrac{1}{2} f_1 \, \tau^*Z, \ad(\tau) \Phi_2^\rho - \Phi_1 \right\> d\theta\\
	& =  \d \left( \frac{i}{2\pi} \int_{S^1} \left< F, \Phi \right\> d\theta \right)
+ \frac{i}{4\pi} \int_{S^1} 2\left\< f_1\tau^*Z, \Phi_1 \right\> + \left\< f_1\tau^*Z, \tau^*Z \right\> d\theta.
\end{align*}
Therefore, $\tau^* R - d\epsilon$ is equal to
$$
\d \left( \frac{i}{4\pi} \int_{S^1} \left\< A, \partial A \right\> - 2 \left\< F, \Phi \right\> d\theta \right) - \frac{i}{4\pi} \int_{S^1} 2\left\< f_1\tau^*Z, \Phi_1 \right\> + \left\< f_1\tau^*Z, \tau^*Z \right\> d\theta.
$$
So it is enough to find a $B_2 \in \Omega^2(P)$ such that
$$
\d B_2 = \frac{i}{4\pi} \int_{S^1} 2\left\< f_1\tau^*Z, \Phi_1 \right\> + \left\< f_1\tau^*Z, \tau^*Z \right\> d\theta.
$$
It is easy to check that the form
$$
\frac{i}{4\pi} \int_{S^1} \left\< \Phi, f \Phi \right\> d\theta
$$
satisfies this. Therefore, we have
\begin{proposition}\label{P:curving}
A curving for the lifting bundle gerbe is given by
$$
B = \frac{i}{4\pi} \int_{S^1} \left\<A, \partial A \right\> - 2\< F + \tfrac{1}{2}f\Phi, \Phi\> \, d\theta.
$$
\end{proposition}

In \cite{Gomi:2003}, Gomi gives another method for calculating a curving for the lifting bundle gerbe, that of \emph{reduced splittings}. In the interest of brevity we omit a discussion of these results, however we note that a reduced splitting for the $\LGS$-bundle which leads to the same curving as in Proposition \ref{P:curving} can be found in \cite{Vozzo:PhD}.

\subsubsection{The real string class of an $\LGS$-bundle}

The last step now that we have found a curving for the lifting bundle gerbe is to calculate the $3$-curvature $H = dB.$ Then $H/2\pi i$ is integral and represents the real image of the Dixmier-Douady class of $\tau^*(\widehat{LG} \rtimes_{\hat\rho} S^1)$ (and hence the obstruction to lifting $P$). We have
$$
dB = \frac{i}{2\pi} \int_{S^1} \left\< dA, \partial A\right\> - \left\< dF, \Phi\right\> - \left\< F, d\Phi\right\> - \left\< d\Phi, f\Phi \right\> d\theta.
$$
To proceed further, we require the Bianchi identity for $(F, f).$ Note that
$$
d(F, f) = ( [dA, A] - f \wedge \partial A + a \wedge \partial (dA) , d^2 a).
$$
In particular, this means that
$$
dF = [F, A] - f \wedge \partial A + a \wedge \partial F,
$$
since $ [F,A] = [dA, A] -[a\wedge \partial A,A]$ and $ a\wedge \partial F  = a \wedge \partial (dA) + [a\wedge \partial A, A].$
Using this, and the fact that $\int_{S^1}\<[A, A], \partial A\> d\theta$ and  $\< a\wedge \partial A, \partial A\>$ both vanish (so that $\int_{S^1} \<dA, \partial A\> d \theta = \int_{S^1} \<F, \partial A\> d \theta$), the expression for $dB$ becomes
$$
dB = \frac{i}{2\pi} \int_{S^1} \left\< F+ f\Phi, \partial A - [A, \Phi]  + a \partial \Phi -   d\Phi\right\> d\theta.\\
$$
If we define the covariant derivative of $\Phi$ by
$$
\nabla \Phi = d\Phi + [A, \Phi] - \partial A - a \partial \Phi,
$$
then one can easily check that it is (twisted) equivariant for the adjoint action. That is,
$$
\nabla \Phi (X (\gamma, \phi)) = \rho_{-\phi} \left(\ad(\gamma^{-1}) \nabla \Phi (X) \right),
$$
for any tangent vector $X.$ The same is true for the quantity $F + f \Phi,$ and so $H= dB$ descends to a form on $M.$ Thus we have proven

\begin{theorem}\label{T:LGxS^1string}
Let $P \to M$ be a principal $LG\rtimes_\rho S^1$-bundle and let $\Phi$ be a Higgs field for $P$ and $(A,a)$ be a connection for $P$ with curvature $(F,f).$ Then the real string class of $P$,  is represented in de Rham cohomology by
$$
-\frac{1}{4\pi^2}\int_{S^1} \langle F+ f\Phi, \nabla\Phi \rangle \,d\theta,
$$
where
$$
\nabla\Phi = d\Phi + [A,\Phi] - \partial A - a\partial \Phi.
$$
\end{theorem}

\begin{remark}
Recall from Remark \ref{rem:reduced} that if $P(S^1) \to Y$ is trivial then $P$ has a reduction to 
an $LG$-bundle. In such a case we choose a flat connection so that  $a = 0$ and  hence $f = 0$. 
The formulae we have derived for the real string class then reduces to that in \cite{Murray:2003} which 
we have given in Theorem \eqref{T:MS031}.
\end{remark}

\subsection{The real string class and the first Pontrjagyn class}\label{SS:string class and p1}

Recall that, in the case of $LG$-bundles, Theorem \ref{T:MS032} related the real string class to the Pontrjagyn class of the corresponding $G$-bundle. We have already seen that Higgs fields for $\LGS$-bundles play an integral role in the caloron correspondence for $G$-bundles over circle bundles and $\LGS$-bundles. In particular, a connection on the $G$-bundle and circle bundle corresponds to a connection and Higgs field on the $\LGS$-bundle. Therefore, the appearance of Higgs fields in our expression for the real string class of an $\LGS$-bundle suggests that there should be a result similar to Theorem \ref{T:MS032}. In particular, we have the following theorem

\begin{theorem}\label{T:LGxS^1Pont}
Let $P \to M$ be a principal $LG\rtimes_\rho S^1$-bundle and $\wt P \to Y \to M$ be the corresponding $G$-bundle over an $S^1$-bundle. Then the real string class of $P$ is given by the integration over the fibre of the first Pontrjagyn class of $\wt P$. That is,
$$
s(P) = \int_{S^1} p_1 (\wt P).
$$
\end{theorem}

\begin{proof}

We prove this by calculating the integral of the first Pontrjagyn class of $\wt P$  over the fibres of $Y \to M$.

Recall that the first Pontrjagyn class is given by
$$
p_1 = -\frac{1}{8\pi^2} \<\tilde{F},\tilde{F}\>,
$$
where $\tilde{F} = d\tilde{A} + \tfrac{1}{2} [\tilde{A},\tilde{A}]$ is the curvature of a connection $\tilde{A}$ on $\wt P$. Recall that in Proposition \ref{P:caloron connection} we gave a formula for the connection on $\wt P$ corresponding to the pair $((A, a), \Phi)$ on $P$. The curvature of this connection is given by
$$
\tilde{F} = \ad(g^{-1})\left( F + f\Phi + \nabla \Phi \wedge (a + d\theta)\right).
$$
So the first Pontrjagyn class is
$$
p_1 = -\frac{1}{8\pi^2} \Big( \left\<  F + f\Phi ,  F + f\Phi \right\> 
-2 \left\<  F + f\Phi , \nabla \Phi \wedge a \right\> 
-2 \left\< F + f\Phi , \nabla \Phi \right\> d\theta \Big).
$$
Thus, integrating $p_1$ over the fibre, we get
$$
-\frac{1}{4\pi^2}\int_{S^1} \langle F+ f\Phi, \nabla\Phi \rangle \, d\theta,
$$
which is the expression from Theorem \ref{T:LGxS^1string}.

\end{proof}

\end{document}